\documentclass[10pt]{amsart}
\usepackage{amssymb,amsmath}
\usepackage{enumerate}
\usepackage{graphicx}
\usepackage{color}
\usepackage{mathrsfs}

\newtheorem{theorem}{Theorem}[section]

\newtheorem{lemma}[theorem]{Lemma}
\newtheorem{corollary}[theorem]{Corollary}
\theoremstyle{definition}

\newtheorem{example}[theorem]{Example}
\newtheorem{remark}[theorem]{Remark}

\usepackage{accents}

\numberwithin{equation}{section}

\begin{document}
\title{Functional inequalities for forward and backward diffusions}
\author{Daniel Bartl \and Ludovic Tangpi}
\thanks{Vienna University, Faculty of Mathematics,  daniel.bartl@univie.ac.at; \\  
Princeton University, ORFE, ludovic.tangpi@princeton.edu.}	
\date{\today}

\begin{abstract}
In this article we derive Talagrand's $T_2$ inequality on the path space w.r.t.\ the maximum norm for various stochastic processes, including solutions of one-dimensional stochastic differential equations with measurable drifts, backward stochastic differential equations, and the value process of optimal stopping problems.

The proofs do not make use of the Girsanov method, but of pathwise arguments.
These are used to show that all our processes of interest are Lipschitz transformations of processes which are known to satisfy desired functional inequalities.
\end{abstract}

\maketitle
\setcounter{equation}{0}


\section{Introduction and main results}
\subsection{Notation}

Let $(\Omega, \mathcal{F}, (\mathcal{F}_t),P)$ be the canonical space of a $d$-dimensional Brownian motion $W$ equipped with the $P$-completion of the filtration $\sigma(W_s : s\le t)$ generated by $W$.
That is, $\Omega=C([0,T],\mathbb{R}^d)$ endowed with the maximum norm, $W_t(\omega)=\omega(t)$, and $P$ is the Wiener measure.
For $p\in[1,\infty)$ and $\mu,\nu\in\mathcal{P}(\Omega)$ (the set of all Borel  probability on $\Omega$) define the $p$-Wasserstein distance and the relative entropy by
\begin{equation*}
	\mathcal{W}_p(\mu,\nu) := \Big(\inf_\pi \int_{\Omega\times\Omega} \|\omega-\eta\|_\infty^p \,\pi(d\omega,d\eta)\Big)^{1/p}
	\quad\text{and}\quad
	H(\nu|\mu) :=\int_\Omega \frac{d\nu}{d\mu}\log\frac{d\nu}{d\mu}\,d\mu,
\end{equation*}
where the infimum is taken over all couplings $\pi$ (that is, probability measures on the product with first marginal $\mu$ and second marginal $\nu$) and we used the convention $d\nu/d\mu=+\infty$ if $\nu$ is not absolutely continuous w.r.t.~$\mu$.
Recall that the quadratic transportation inequality (sometimes called Talagrand's inequality) reads as
\[ \mu\text{ satisfies } T_2(C) \text{ if } \mathcal{W}_2(\mu,\nu)\leq \sqrt{C H(\nu|\mu)} \text{ for all } \nu\in\mathcal{P}(\Omega).\]

The validity of such inequalities has several (deep) consequences, for instance to the concentration of measure phenomenon, the isoperimetric problem and various problems of probability in high dimensions.
We refer the reader e.g.\ to \cite{Ledoux01,Vil2,Talagrand96,Marton96} for an overview and applications.
Let us for instance mention a result by Gozlan, see \cite[Theorem 1.3]{Gozlan09}, who showed that $T_2(C)$ is equivalent to the dimension-free concentration
\begin{equation}
\label{eq:dim-free conc}
	\mu^n\Big(\Big|F - \int F d\mu^n\Big|>x\Big) \leq 2\exp(-cx^2)
\end{equation}
for all $x>0$, $n\geq1$ and some constant $c>0$, where $\mu^n$ is the $n$-fold product of $\mu$ and $F\colon\Omega^n\to \mathbb{R}$ is a $1$-Lipschitz function w.r.t.\ the $l_2$-norm on $\Omega^n$.

\subsection{Main results}

In this work, we prove the validity of $T_2$ for various stochastic processes evolving forward and backward in time.
Let us present our principal contribution; the proofs are postponed to later sections along with applications and consequences.

\subsubsection{Optimal Stopping}\hfill\\ 
Our first result concerns the value process of an optimal stopping problem.

\begin{theorem}[Optimal Stopping] 
\label{thm:stopping}
	Let $\Gamma\colon [0,T]\times \Omega \to\mathbb{R}$ be an adapted process with continuous paths such that $\Gamma_t$ is $L_\Gamma$-Lipschitz for every $t\in[0,T]$ and denote by 
	\begin{equation}
	\label{eq:S.stopp}  
	\tag{OptStop}
	S_t:=\mathop{\mathrm{ess.sup}}_{\tau  \text{ is stopping time, } t\leq \tau\leq T} E[\Gamma_{\tau}|\mathcal{F}_t]
	\end{equation}
	for $t\in[0,T]$ the value process of the optimal stopping problem of $\Gamma$.
	Then $S$ has continuous paths and
	\[ \text{the law $\mu^s$ of } S \text{ satisfies } T_2(C_s) \]
	with $C_s:=2L^2_\Gamma$.
\end{theorem}

More generally, Brownian motion can be replaced by an appropriate backward diffusion, see Corollary \ref{cor:stopping bsde}.

In many applications it is interesting to approximate the law $\mu^s$ of $S$.
Denote by $\mu_N:=\mu_N^s$ the empirical measure associated to $\mu:=\mu^s$, that is, we fix $P^\infty$ the infinite product of $P$ under which i.i.d.\ random variables $(S^n)_{n\in\mathbb{N}}$ with distribution $\mu$ are defined, and we put $\mu_N:=\tfrac{1}{N}\sum_{n=1}^N \delta_{S^n}$ for every $N\geq 1$.
Applying \eqref{eq:dim-free conc}  to the 1-Lipschitz function $F(y):=\sqrt{N}\mathcal{W}_2(\tfrac{1}{N}\sum_{n=1}^N\delta_{y^n},\mu)$ gives the following concentration property of Wasserstein distance between the true and the empirical measure:

\begin{corollary}
\label{cor:wasserstein.conc}
	In the setting of Theorem \ref{thm:stopping} there is $c>0$ such that 
	\[P^\infty\Big( \Big| \mathcal{W}_2(\mu,\mu_N) -E_{P^\infty}[\mathcal{W}_2(\mu,\mu_N)] \Big| \ge x \Big) \leq 2\exp(-cx^2N)\]
	for every $x>0$ and  $N\geq 1$.
\end{corollary}

Note that the same holds true for the diffusions considered Theorem \ref{thm:bsde.multi.dim}, Theorem \ref{thm:bsde.1.dim}, or Theorem \ref{thm:SDE} below.

Also note that by convergence of $E_{P^\infty}[\mathcal{W}_2(\mu^y,\mu^y_N)]$ to zero, the above implies that there is $c>0$ and, for every $x>0$, some $N_0(x)$ such that 
\[P^\infty(\mathcal{W}_2(\mu,\mu_N)\ge x) \leq 2\exp(-cx^2N)\]
for all $N\geq N_0(x)$.

Another consequence (which could also be shown by simpler methods) is due to the fact that the $T_2$-inequality implies Gaussian concentration \cite[Theorem 22.10]{Vil2}, that is, Theorem \ref{thm:stopping} in particular implies that
\[P( |S_t - E[S_t]| \geq x) \leq 2\exp( -cx^2 )\]
for all $x>0$, where $c>0$ is a constant.
This means that (on a large scale) $E[S_t]$ can be seen as a good proxy for the value of $S_t$.
This is interesting in that  $E[S_t]=\sup_{t\leq \tau\leq T} E[\Gamma_\tau]$ can usually be computed quite efficiently (see e.g.\ \cite{Bec-Che-Jen19}) while  the computation of $S_t$ might be hard.

\subsubsection{Backward diffusions} \hfill\\
We now turn our attention to the backward stochastic differential equation
\begin{equation}
\tag{BSDE}
\label{eq:bsde}
	Y_t = F + \int_t^Tg_u( Y_u, Z_u)\,du - \int_t^TZ_u\,dW_u
	\quad\text{for }t\in[0,T]
\end{equation}
whose solution is given by a pair of processes $(Y,Z)$  on the canonical space $\Omega$ with $Y$ adapted and $Z$ progressive.
We have the following:

\begin{theorem}[$T_2$ for multi-dim BSDE]
\label{thm:bsde.multi.dim}
	Let $m\in\mathbb{N}$ and assume that
	\begin{enumerate}[(A)]
	\item
	\label{b1}
	$g\colon[0, T]\times \Omega \times \mathbb{R}^m \times \mathbb{R}^{m\times d} \to \mathbb{R}^m$ is progressive, $g_t(\cdot,\cdot,\cdot)$ is $L_g$-Lipschitz continuous for every $t\in[0,T]$, and $E[\int_0^T|g_t(\cdot,0,0)|^2\,dt]<\infty$,
	\item 
	$F\colon\Omega\to \mathbb{R}^m$ is $L_F$-Lipschitz continuous.
	\label{b2}
	\end{enumerate}
	Then there exists a unique solution $(Y,Z)$ of \eqref{eq:bsde} and 
	\[ \text{the law }\mu^y \text{ of } Y \text{ satisfies } T_2(C_y)\]
	with $C_y: = 2(L_F+TL_g)^2e^{2TL_g}$.
\end{theorem}
	
Remarkably, the constant $C_y$ in Theorem \ref{thm:bsde.multi.dim} does not depend on $m$ and $d$, suggesting that the result can be extended to infinite dimensional BSDEs  (e.g.\ BSDEs on Hilbert spaces analyzed in \cite{Fuh-Tess02}).
We will not take up this task here.
As an (rather direct) application of Theorem \ref{thm:bsde.multi.dim}, we derive in Corollary \ref{cor:ineq y Q-martingale} a transportation inequality for laws of martingales, thus extending a result by Pal \cite{Pal12}.
Moreover, we will show in Section \ref{sec:ineq z} that under additional conditions pertaining to the regularity of $g$, functional inequalities can also be deduced for the law of the control process $Z$.

When $Y$ is one-dimensional (but the Brownian motion still $d$-dimensional) the regularity conditions on $g$ can be weaken as follows:

\begin{theorem}[$T_2$ for 1-dim BSDE]
\label{thm:bsde.1.dim}
	Assume that 
	\begin{enumerate}[(A)]
	\item
	$g\colon[0, T] \times \mathbb{R}^d \to \mathbb{R}_+$ is Borel measurable and convex in the last variable, 
	\item
	$g_t(z)\le C(1 + |z|^2)$ for all $z \in \mathbb{R}^d$ and for some constant $C>0$.
	\item 
	$\inf_{t\in[0,T]} g_t(z)/|z|\to\infty$ as $|z|\to\infty$, and
	\item 
	$F\colon\Omega\to \mathbb{R}$ is bounded from below and $L_F$-Lipschitz continuous.
	\end{enumerate}
	Then \eqref{eq:bsde} admits a unique solution $(Y,Z)$ and 
	\[ \text{the law } \mu^y \text{ of } Y \text{ satisfies } T_2(C_y)\]
	with $C_y := 2L_F^2$.
\end{theorem}

\begin{remark}
	When $d=1$, $g = 0$ and $F=\mathrm{id}$, then $Y$ is the Brownian motion, and $C_y=2$ (which is known to be optimal for Brownian motion) showing that the constant $C_y$ in Theorem \ref{thm:bsde.multi.dim} and Theorem \ref{thm:bsde.1.dim} cannot be improved in general.
\end{remark}

\begin{example}
	By Theorem \ref{thm:bsde.1.dim}, the law of the process $Y_t:=\log E[\exp(F)|\mathcal{F}_t]$ satisfies $T(2L_F^2)$ for every $L_F$-Lipschitz continuous function $F$ on $\Omega$  which is bounded from below.
	In fact, it follows from martingale representation and It\^o's formula that there is a progressive process $Z$ such that $(Y,Z)$ solve equation \eqref{eq:bsde} with $g(z)= \frac12|z|^2$ and terminal condition $F$.
\end{example}

BSDEs provide a powerful probabilistic tool to tackle second order nonlinear partial differential equations as first noted by \cite{Pardoux-Peng92}.
They can be seen as a nonlinear generalization of the maximum principle in stochastic control theory.
Moreover, BSDEs have various applications in quantitative finance.
The following is a prime example stated in more generality (and precision) in Section \ref{sec:portfolio}.

\begin{example}[Utility maximization]
\label{exa:finance}
	Let $F\colon \Omega\to\mathbb{R}$ be bounded and $L_F$-Lipschitz continuous, and consider the Black-Scholes dynamics $dS_t=S_t(dt + \,dW_t)$ for stock price.
	The process 
	\[ V_t:=\mathop{\mathrm{ess\,sup}}_{p} E\Big[ U\Big( F- \int_t^T p_u \,\frac{dS_u}{S_u} \Big)\Big|\mathcal{F}_t\Big]
	\qquad\text{where } U(x)=-\exp(-x), \]
	where the supremum is taken over predictable portfolios $p$ subject to some integrability condition, defines the value process of the exponential utility maximization problem in the Black-Scholes market with random endowment $F$.
	
	Then (a suitable transformation of) $V$ and the optimal trading strategy $\pi^\ast$ are characterized by a BSDE and we will see that both satisfy the $T_2$-inequality.
	In particular, concentration of empirical measure as in Corollary \ref{cor:wasserstein.conc} or Gaussian concentration hold, showing for example that the value of the optimal utility and portfolio are concentrated around their mean. 
\end{example}

\subsubsection{Forward diffusions} \hfill\\
Let us finally consider the stochastic differential equation
\begin{equation}
\tag{SDE}
\label{eq:SDE intro}
	X_t = x + \int_0^tb_u(X_u)\,du + \int_0^t\sigma_u(X_u)\,dW_u
	\quad\text{for } t\in[0,T]
\end{equation} 
in dimension $1$ (again, the Brownian motion is $d$-dimensional).

\begin{theorem}[$T_2$ for 1-dim SDE]
\label{thm:SDE}
	Assume that
	\begin{enumerate}[(A)]
	\item
	\label{ass:sde.b}
	$b\colon[0,T]\times \mathbb{R}\to \mathbb{R}$ is Borel measurable and continuously differentiable in the first variable,
	\item
	\label{ass:sde.sigma}
	$\sigma\colon[0,T]\times\mathbb{R}\to \mathbb{R}^d$ is bounded, continuously differentiable in the first variable and $L_\sigma$-Lipschitz continuous in the second variable,
	 and $\sigma\sigma'\ge c>0$ for some constant $c$,
	\item the quantities
	\label{ass:sde.b.divided.sigma}
	$c_1:=\sup_{t\in [0,T]}\|\frac{b_t}{\sigma_t\sigma_t'}(\cdot)\|_{L^1(dx)}$; $c_2:=\sup_{t\in [0,T]}\| \frac{b_t}{\sigma_t\sigma_t'}(\cdot)\|_{L^\infty(dx)}$; $c_3:=\sup_{t\in [0,T]} \|\frac{\partial}{\partial t}\frac{b_t}{\sigma_t\sigma_t'}(\cdot)\|_{L^1(dx)}$ and $c_4:=\|\sup_{t\in [0,T]} \frac{\partial}{\partial t}\frac{b_t}{\sigma_t\sigma_t'}(\cdot)\|_{L^1(dx)}$ are finite.
	\end{enumerate}
	Then, \eqref{eq:SDE intro} admits a unique strong solution and
	\[ \text{the law }\mu^x \text{ of } X \text{ satisfies } T_2(C_x) \]
	with $C_x:=6\exp(c_1+15 \max( c_3\exp(2c_1), \|\sigma\|_\infty c_2\exp(2c_1) + \exp(2c_1)L_\sigma^2))$.
\end{theorem}

\begin{example}
	The above result applies for instance to Langenvin's equations:
	These are stochastic differential equations of the form
\begin{equation}
\label{eq:langenvin}
	dX_t = -U'(X_t)\,dt + \sqrt{2/\lambda}\,dW_t,\quad X_0 = x
\end{equation}
where $U:\mathbb{R}\to \mathbb{R}$ is differentiable and plays the role of a potential, and $\lambda>0$ (we take for simplicity $d=1$).
We will discuss this example further in Section \ref{sec:sde}.
\end{example}

\subsubsection{Variations and extensions} \hfill\\
The arguments leading to our main results for backwards diffusions and the optimal stopping problem (inspired by techniques used in `pathwise control theory' see e.g.\ \cite{Ekr-Kel-Tou-Zha14,Nutz-Zhang15}) consist in showing that the processes under consideration are Lipschitz transforms of Brownian motion.
It is likely that this technique also applies to more general stochastic optimal control problems and that different consequences than the above can be deduced.

To illustrate our point, we derive the logarithmic-Sobolev inequality, which reads as:
\begin{equation*}
	\mu\in\mathcal{P}(\mathbb{R}^m)\text{ satisfies } LSI(C)
	\quad\text{if}\quad
	\mathrm{Ent}_\mu(f) \leq C\int_{\mathbb{R}^m}|\nabla f|^2\,d\mu
\end{equation*}
for every $\mu$-integrable and differentiable function $f\colon\mathbb{R}^m\to \mathbb{R}$.
Here $\mathrm{Ent}_\mu(f):=\int f^2 \log(f^2 / \int f^2\,d\mu) \,d\mu$ is the entropy of $f$ w.r.t.\ $\mu$ with the convention $0/0=0$.

\begin{theorem}[Log-Sobolev]
\label{thm:lsi y}
	In the setting of either Theorem \ref{thm:bsde.multi.dim} or Theorem \ref{thm:bsde.1.dim}, for every $t\in[0,T]$ one has that
	\[  \text{the law $\mu^y_t$ of } Y_t \text{ satisfies } LSI(T C_y)\]
	with the constant $C_y$ given in the respective theorems.
	
	The same holds true in the setting of Theorem \ref{thm:stopping} (if $Y$ above is replaced by $S$).
\end{theorem}

\subsection{Related literature}
Measure concentration is a popular area of modern probability theory.
This is mostly due to its variety of applications, including (and certainly not restricted to) model selection, random algorithms, quantitative finance and statistics \cite{Massart,Bou-Gab-Mas,Bou-Tho,Lac-concent,ConcenRM,Dubh-Panc,Pal-Shkol14}.
It is the works of Marton \cite{Marton96} and Talagrand \cite{Talagrand96} that first underlined the relevance of transportation inequalities in the description of the concentration of measure phenomenon.
Transportation inequalities are also related to Poincar\'e inequality, log-Sobolev inequality and hypercontractivity, see \cite{Otto-Vil00,Bob-Gen-Led01}.

Talagrand proved the validity of $T_2$ for the multidimensional Gaussian distribution with optimal constant $C = 2$.
His work was then extended to Wiener measure on the path space in \cite{Fey-Ust04}.
Transportations inequalities for laws of (forward) SDEs have been extensively studied.
For the case of equation driven by Brownian motion, see \cite{Dje-Gui-Wu,Uestuenel12,Pal12} and for SDEs driven by fractional Brownian motions or Gaussian processes refer to \cite{Sau,Riedel}.
All the aforementioned works on SDEs assume that the coefficients are Lipschitz-continuous or satisfy a dissipative condition.
Note however the exception of \cite{Abakirova} who derives versions of the Poincar\'e and log-Sobolev inequalities of the so-called skew Brownian motions, which can be seen as solutions of SDEs with local time of the unknown.

Regarding Transportations inequalities for backward SDEs, to the best of the authors' knowledge the only work on the subject is the paper \cite{Bah-Bou-Mou19} available online since August 2019.
It uses the Girsanov transform technique of \cite{Dje-Gui-Wu} to derive quadratic transportation inequalities for laws of one-dimensional BSDEs with bounded coefficients, and Lipschitz  continuous generators.

\subsection{Organization of the paper}
We organize the rest of the paper as follows: The next section is dedicated to the proofs of Theorems \ref{thm:stopping} and \ref{thm:bsde.multi.dim}, we also discuss various extensions and applications, including functional inequalities for laws of martingales.
In Section \ref{sec:ineq bsde}, we present the proof of Theorem \ref{thm:bsde.1.dim}.
The application to portfolio optimization alluded in the introduction is presented in more details.
We prove Talagrand inequality for SDEs with measurable drifts in Section \ref{sec:sde} and conclude with the analysis of the logarithmic-Sobolev inequality.

\section{The proofs of Theorem \ref{thm:bsde.multi.dim} and Theorem \ref{thm:stopping}}

The strategy behind the proofs is to show that the objects of interest are in fact obtained through Lipschitz transformations of Brownian motion.
The latter is known to satisfy the $T_2$-inequality, see \cite[Theorem 3.1]{Fey-Ust04}.
For this reason, the following lemma on the stability of transportation inequalities under push-forward by Lipschitz maps (taken from \cite{Dje-Gui-Wu}) is fundamental and stated separately.

\begin{lemma}[\text{\cite[Lemma 2.1]{Dje-Gui-Wu}}]
\label{lem:stab push-forw}
	Assume that $\mu$ satisfies $T_2(C)$ and let $\psi\colon \Omega\to\Omega$ be $\mu$-almost surely $L_\psi$-Lipschitz.
	Then the push-forward $\psi_\ast\mu$ satisfies $T_2(CL_\psi^2)$.
\end{lemma}

The proofs need some notational preparation, introduced below.
For $t\in [0,T]$ denote by 
\[\Omega^t:= \{\gamma \in C([t,T],\mathbb{R}^d) : \gamma(t)=0\}\] 
the shifted canonical space, by $W^t$ the canonical process on $\Omega^t$, by $P^t$ the Wiener measure on $\Omega^t$, and by $(\mathcal{F}^t_s)_{s\in [t,T]}$ the $P^t$-completion of the natural filtration of $W^t$. 
For $\omega\in \Omega$, $t\in[0,T]$, and $\gamma\in\Omega^t$ define the concatenation $\omega\otimes_t\gamma\in\Omega$ via
\begin{equation*}
	(\omega\otimes_t\gamma)(s) :=
	\begin{cases}
	\omega(s) &\text{if } s\in[0,t),\\
	\gamma(s) + \omega(t) &\text{if } s\in[t,T].
	\end{cases}
\end{equation*}
Further, for a function $X\colon\Omega\times[0,T]\to\mathbb{R}$ and fixed $(t,\omega)\in[0,T]\times \Omega$, define its shifts by 
\[ X^{t,\omega}\colon \Omega^t\times [t,T]\to\mathbb{R},\quad X^{t,\omega}_s(\gamma) := X_s(\omega\otimes_t\gamma). \]
Similar notation is applied to a function $X\colon \Omega\to\mathbb{R}$ or a function $g\colon \Omega\times[0,T]\times A\to\mathbb{R}$, where $A$ is an arbitrary space, that is, $X^{t,\omega}(\gamma)=X(\omega\otimes_t\gamma)$ or $g^{t,\omega}_s(\gamma,a):=g_s(\omega\otimes_t\gamma,a)$.

Note that, using the above notation, one has 
\[ E[X|\mathcal{F}_t](\omega)
=\int_{\Omega^t} X^{t,\omega}(\gamma)\,P^t(d\gamma)
=:E_{P^t(d\gamma)}[X^{t,\omega}(\gamma)] 
=:E_{P^t}[X^{t,\omega}]\]
for $P$-almost all $\omega\in\Omega$.

\subsection{Proof of Theorem \ref{thm:bsde.multi.dim} and first consequences}
\begin{proof}[Proof of Theorem \ref{thm:bsde.multi.dim}]
	It follows from the work of Pardoux $\&$ Peng \cite[Theorem 3.1]{PP90} that the equation \eqref{eq:bsde} admits a unique solution $(Y,Z)$ such that $Z$ is square integrable and $Y$ has ($P$-almost surely) continuous paths, that is, $Y(\omega)\in C([0,T],\mathbb{R}^m)$ for ($P$-almost every) $\omega\in\Omega$.
	Using arguments close in spirit to \cite{Ekr-Kel-Tou-Zha14}, we will show that the function $\omega\mapsto Y(\omega)$ is Lipschitz continuous.
	
	Let $t\in[0,T]$.
	By Lemma \ref{lem:shifed.bsde} below there is a $P$-zero set $N\subset\Omega$ such that for $\omega\in N^c$ one has $Y_t(\omega)=Y_t^{t,\omega}$ $P^t$-almost surely and the pair $(Y^{t,\omega}_r, Z^{t,\omega}_r)_{r\in [t,T]}$ satisfies
	\begin{equation}
		Y^{t, \omega}_r = F^{t, \omega} + \int_r^Tg^{t,\omega}_u(W^t, Y^{t, \omega}_u, Z^{t, \omega}_u)\,du - \int_r^TZ^{t, \omega}_u\,dW^t_u , \quad P^t\text{-a.s.\ $r\in [t,T]$}.
	\end{equation}
	From now on fix $\omega, \eta \in N^c$ and $t\in [0,T]$.
	For $r\in[t,T]$ define
	\begin{align*}
		\delta Y_r&:= Y^{t, \omega}_r - Y^{t, \eta}_r,\\
		\delta Z_r&:= Z^{t,\omega}_r - Z^{t, \eta}_r,	\\
		\delta g_r&:= g^{t, \omega}_r(W^t, Y^{t, \omega}_r, Z^{t, \omega}_r) - g^{t, \eta}_r(W^t, Y^{t, \omega}_r, Z^{t, \omega}_r).
	\end{align*}
	As $(y,z)\mapsto g_r(\omega,y,z)$ is Lipschitz and therefore Lebesgue almost surely differentiable, it follows that
	\begin{align*}
	&g^{t, \omega}_r(W^t, Y^{t, \omega}_r, Z^{t, \omega}_r) - g^{t, \eta}_r(W^t, Y^{t, \eta}_r, Z^{t, \eta}_r)  \\
	&= \delta g_r + \int_0^1 
	\begin{pmatrix} 
		\partial_y g_r^{t,\eta}(W^t, Y^{t, \omega}_r -a \delta Y_r , Z^{t, \omega}_r - a \delta Z_r) \\
		\partial_z g_r^{t,\eta}(W^t, Y^{t, \omega}_r -a \delta Y_r , Z^{t, \omega}_r - a \delta Z_r)
	\end{pmatrix}^\top 
	\begin{pmatrix} 
		\delta Y_r\\
		\delta Z_r
	\end{pmatrix} 
		\,da\\
	&:=\delta g_r + \beta_r\delta Y_r + q_r \delta Z_r. 
	\end{align*}
	Note that the progressive processes $\beta$ and $q$ are bounded by $L_g$.
	Moreover, the pair $(\delta Y, \delta Z)$ solves the linear equation 
	\begin{equation*}
		\delta Y_r = \delta Y_T + \int_r^T \delta g_u + \beta_u \delta Y_u + q_u \delta Z_u\,du - \int_r^T \delta Z_u\,dW_u^t, \quad P^t\text{-a.s.}
	\end{equation*}
	for $r\in [t,T]$, and a standard computation as in \cite[Theorem 1.1]{karoui01} reveals that 
	\begin{align*}
		\delta Y_r&= e^{\int_r^T\beta_u\,du}\delta Y_T + \int_r^Te^{\int_r^u\beta_s\,ds} \delta g_u \,du- \int_r^T e^{\int_r^u\beta_s\,ds} Z_u(dW^t_u - q_u\,du) \quad P^t\text{a.s.}
	\end{align*}
	For $r\in [t, T]$, define
	\begin{equation*}
		\Gamma_r:= \exp\Big(\int_t^r q_u\,dW_u - \frac 12 \int_t^r |q_u|^2\,du \Big) \exp\Big( \int_t^r \beta_u\,du\Big).
	\end{equation*}
	By Girsanov's theorem, taking the expectation with respect to the (shifted) measure $P^t$ yields
	\begin{equation}
	\label{eq:indentity with Gamma}
		|\delta Y_t| = \Big| E_{P^t}\Big[\Gamma_T\delta Y_T + \int_t^T\Gamma_r\delta g_r\,dr \Big] \Big|
	\end{equation}
	$P^t$-almost surely.
	Moreover, by Lipschitz continuity of $g$, it holds that
	\begin{align*}
		|\delta g_r(\gamma)| 
		&\leq L_g \|\omega\otimes_t \gamma- \eta\otimes_t \gamma \|_\infty\\
		&\leq L_g \|\omega-\eta\|_\infty
	\end{align*}
	for all $\gamma\in\Omega^t$.
	Thus, as $E_{P^t}[\Gamma_r]\leq \exp(T L_g)$ for every $r\in[t,T]$, it follows from \eqref{eq:indentity with Gamma} that
	\begin{align*}
	|\delta Y_t| 
	&\leq L_F \|\omega-\eta\|_\infty E_{P^t}[\Gamma_T]  + L_g \|\omega-\eta\|_\infty E_{P^t}\Big[\int_t^T \Gamma_r \,dr\Big]\\
	&\leq (L_F+TL_g)\exp(TL_g) \|\omega-\eta\|_\infty
	\end{align*}
	$P^t$-almost surely.
	
	As $\omega,\eta\in N^c$ and $t\in[0,T]$ were arbitrary and $P(N)=0$, this shows that 
	\[Y\colon\Omega\to C([0,T],\mathbb{R}^m) \quad\text{is } (L_F+TL_g)e^{TL_g}\text{-Lipschitz}.\]
	Now recall that by \cite[Theorem 3.1]{Fey-Ust04}, the probability measure $P$ (the law of the Wiener process) satisfies $T_2(2)$.
	Hence, the result follows by Lemma \ref{lem:stab push-forw}.
\end{proof}

\begin{lemma}	
\label{lem:shifed.bsde}
	For $P$-almost all $\omega\in\Omega$ it holds that
	\[Y^{t, \omega}_r = F^{t, \omega} + \int_r^Tg^{t,\omega}_u(W^t, Y^{t, \omega}_u, Z^{t, \omega}_u)\,du - \int_r^TZ^{t, \omega}_u\,dW^t_u , \quad P^t\text{-a.s.\ $r\in [t,T]$}\]
	and $Y_t(\omega)=Y_t^{t,\omega}$ $P^t$-almost surely.
\end{lemma}
\begin{proof}
	Let $r\ge t$ be fixed and denote by $N_1^c$ the set of all $\omega\in\Omega$ such that \eqref{eq:bsde} holds true so that $P(N_1)=0$.
	Then, for every $\omega\in\Omega$ and $\gamma\in\Omega^t$ such that $\omega\otimes_t\gamma\in N^c_1$, unwrapping the definitions of $Y^{t,\omega}$, $Z^{t, \omega}$ and \eqref{eq:bsde} it holds that
	\begin{align}
	\label{eq:path.shift.bsde}
	Y^{t, \omega}_r(\gamma) 
	=F^{t, \omega}(\gamma) + \int_r^T g^{t,\omega}_u(\gamma, Y^{t, \omega}_u(\gamma), Z^{t, \omega}_u(\gamma))\,du - \Big( \int_r^T Z_u\,dW_u\Big)(\omega\otimes_t\gamma).
	\end{align}
	As the law of the concatenation $\Omega\times\Omega^t\ni(\omega,\eta)\mapsto\omega\otimes_t\eta\in\Omega$ under $P\otimes P^t$ equals $P$, one has 
	\begin{align*}
	P( N_2^c) 
	&=P(N_1^c)
	=1, \quad\text{where}\\
	N_2^c&:=\{\omega\in\Omega : P^t(\gamma\in\Omega^t : \omega\otimes_t \gamma\in N^c_1)=1\}.
	\end{align*}
	For every $\omega\in N^c_2$ we have that \eqref{eq:path.shift.bsde} holds for $P^t$-almost all $\gamma\in\Omega^t$.
	Thus we are left to show that for $P$-almost all $\omega\in\Omega$, one has that 
	\begin{align}
	\label{eq:int.path.change}
	\Big(\int_r^T Z_u\,dW_u\Big)(\omega\otimes_t\gamma)
	=\Big( \int_r^T Z^{t,\omega}_u\,dW^t_u\Big)(\gamma)
	\end{align}
	for $P^t$-almost all $\gamma\in\Omega^t$ and all $r\in[t,T]$.

	In case that $Z$ is a simple processes, the $dW$ and $dW^t$-integrals are just finite sums.
	Then, as $r\geq t$, only increments of $\gamma$ appear in either sums and it follows that \eqref{eq:int.path.change} holds true for all $\omega\in\Omega$ and $\gamma\in\Omega^t$.
	In the general case, approximate $Z$ in $L^2(P\otimes du)$ by simple integrands $Z^n$ which are progressive w.r.t.~the raw filtration (in particular, $Z^{n,t,\omega}$ is progressive w.r.t.~$\mathcal{F}^t$ and the same holds true for the limit).
	Using once more that the law of the concatenation under $P\otimes P^t$ equals $P$, one obtains that
	\begin{align*}
	&E_{P(d\omega)}\Big[  E_{P^t(d\gamma)}\Big[ \int_t^T |Z^{n,t,\omega}_u(\gamma)- Z^{t,\omega}_u(\gamma)|^2\,du\Big]\Big]\\
	&=E_{P(d\omega)}\Big[ \int_t^T |Z^n_u(\omega)- Z_u(\omega)|^2\,du\Big]
	\to 0.
	\end{align*} 
	After passing to a subsequence, one may assume that the inner expectation converges to 0 for $P$-almost all $\omega\in\Omega$, that is $Z^{n,t,\omega}\to Z^{t,\omega}$ in $L^2(P^t\otimes du)$.
	The triangle inequality, Ito's isometry, and the fact that the law of $\cdot\otimes_t\cdot$ under $P\otimes P^t$ equals $P$ then show that
	\[E_{P(d\omega)}\Big[E_{P^t(d\gamma)}\Big[\Big|\Big(\int_r^T Z_u\,dW_u\Big)(\omega\otimes_t\gamma)- \Big(\int_r^T Z^{t,\omega}_u\,dW^t_u\Big)(\gamma)\Big|^2\Big]\Big]=0,\]
	which implies \eqref{eq:int.path.change}.
	
	To complete the proof, we are left to prove that $Y_t(\omega)=Y_t^{t,\omega}$ $P^t$-almost surely.
	This is a consequence of Blumenthal's 0-1 law, i.e.~conditioning on the right-continuous filtration is up to $P$-zero sets the same as conditioning on the raw filtration.
	The raw filtration at time $t$ is generated by paths up to time $t$, hence $Y_t(\omega)=Y_t(\eta)$ for all $\omega,\eta\in\Omega$ such that $\omega=\eta$ on $[0,t]$.
\end{proof}

In the next corollary we show that in our Brownian filtration, laws of martingales satisfy transportation inequalities:
This corollary as well as the subsequent subsection use the notion of Malliavin derivative. 
We refer the reader to \cite{Nualart2006} for an introduction to this topic and the (little bit of) Malliavin calculus used in the article.

\begin{corollary}
\label{cor:ineq y Q-martingale}
	Let $q\colon [0,T]\times \Omega\to \mathbb{R}^{d\times m}$ be a bounded progressive process and let $M$ be an $m$-dimensional martingale under the probability measure $Q^q := \mathrm{Law}(W + \int q\,du)$.
	If $M_T$ and $q$ are both Lipschitz continuous in $\Omega$, then the law of $M$ satisfies $T_2(C)$ for some constant $C>0$ depending on $d$.
\end{corollary}
Here `$q$ is Lipschitz continuous in $\Omega$' means that there is a constant $L_q$ such that $|q_t(\omega)-q_t(\eta)|\leq L_q \|\omega-\eta\|_\infty$ for all $\omega,\eta\in\Omega$ and $t\in[0,T]$.
\begin{proof}
	By Lipschitz continuity, $M_T$ is square integrable.
	Thus, it follows from martingale representation and Girsanov's theorem that the process $M$ satisfies
	\begin{equation}
	\label{eq:martingale bsde}
		M_t = M_T + \int_t^TZ_uq_u\,du - \int_t^TZ_u\,dW_u
	\end{equation}
	for some progressive, square integrable process $Z$.
	In particular, $(M,Z)$ satisfies the equation \eqref{eq:bsde} with generator $g$ defined by $g_t(\omega,y,z):=q_t(\omega)z$ and terminal condition $F:= M_T$.
	
	As $M_T$ is Lipschitz continuous and $q$ of bounded Malliavin derivative (this follows from the Lipschitz assumption on $q$, see e.g.\ \cite[Proposition 3.2]{Che-Nam}), we have by \cite[Theorem 2.2]{Che-Nam} that $Z$ must be bounded, say by $C$.
	Thus, $M$ also satisfies \eqref{eq:martingale bsde} with the now Lipschitz-continuous generator $\tilde g_t(z):= g_t(z)1_{\{|z|\le C\}} + g(z C/|z|)1_{\{|z|>C\}}$.
	The result then follows from Theorem \ref{thm:bsde.multi.dim}.
\end{proof}

\subsection{Proof of Theorem \ref{thm:stopping}}
We shift our focus on the proof of Theorem \ref{thm:stopping}.
Denote by $\mathcal{T}^t$ the set of all stopping times $\sigma\colon \Omega^t\to[t,T]$, that is, $\{\sigma\leq s\}\in \mathcal{F}^t_s$ for all $s\in[t,T]$.

\begin{lemma}
\label{lem:ess.sup.sup}
	It holds that
	\[ \mathop{\mathrm{ess.sup}}_{\tau\text{ is stopping time, }t\leq \tau\leq T} E[ \Gamma_\tau |\mathcal{F}_t](\omega)
	=\sup_{\sigma\in\mathcal{T}^t} E_{P^t}[\Gamma^{t,\omega}_{\sigma}] \]
	for $P$-almost all $\omega\in\Omega$.
\end{lemma}
\begin{proof}
	In a first step, note that one may restrict everywhere to stopping times w.r.t.\ the raw filtration.
	Indeed, for a general stopping time $\tau$ there is a stopping time $\tau'$ w.r.t.\ the right-continuous version of the raw filtration such that $\tau=\tau'$ almost surely.
	Further, integrability and pathwise continuity of $\Gamma$ guarantee that $E[\Gamma_{\min(\tau'+\varepsilon,T)}]\to E[\Gamma_{\tau'}]$ as $\varepsilon\to 0$.
	It remains to notice that $\min(\tau'+\varepsilon,T)$ is a stopping time w.r.t.\ the raw filtration for every $\varepsilon>0$.
	The same arguments apply to conditional expectations.

	We start by showing that the left hand side is smaller than the right hand side.
	To that end, by definition of the essential supremum, there exists a sequence $(\tau_n)$ of stopping times with values in $[t,T]$ such that, $P$-almost surely, the left hand side equals $\sup_n E[\Gamma_{\tau_n}|\mathcal{F}_t]$.
	For every $n$ and $\omega\in\Omega$ one has that $\tau_n^{t,\omega}\in\mathcal{T}^t$, which shows that
	\begin{align*}
	E[\Gamma_{\tau_n}|\mathcal{F}_t](\omega)
	&= E_{P^t(d\gamma)}[\Gamma^{t,\omega}_{\tau_n^{t,\omega}(\gamma)}(\gamma)]\\ 
	&\leq \sup_{\sigma\in\mathcal{T}^t} E_{P^t(d\gamma)}[\Gamma^{t,\omega}_{\sigma(\gamma)}(\gamma)]
	\end{align*}
	for $P$-almost all $\omega$.
	Taking the countable supremum thus yields the first claim.
	
	As for the reverse inequality, assume first that 
	\[ \Gamma^{\omega,t}=\sum_n f^n 1_{A^n}(\omega) \quad\text{for every } \omega\in\Omega,\]
	where $(A^n)$ is a $\mathcal{F}_t^0$-measurable partition of $\Omega$ and $f^n$ are functions from $\Omega^t\times[t,T]$ to $\mathbb{R}$.
	Then 
	\[\sup_{\sigma\in\mathcal{T}^t} E_{P^t}[\Gamma^{t,\omega}_{\sigma}]
	=\sum_n 1_{A^n}(\omega) \sup_{\sigma\in\mathcal{T}^t} E_{P^t}[f^n_{\sigma}] \]
	for every $\omega\in\Omega$.
	Now, let $\varepsilon>0$ be fixed and, for every $n$, pick some $\sigma^n\in\mathcal{T}^t$ which achieves the supremum above up to an error of $\varepsilon>0$.
	Define 
	\[\tau\colon\Omega\to[t,T] \qquad\tau(\omega):=\sum_n 1_{A^n}(\omega) \sigma^n(\omega|_{[t,T]}-\omega(t)),\]
	that is, one has $\tau(\omega\otimes_t\gamma)=\sum_n 1_{A^n}(\omega) \sigma^n(\gamma)$.
	Then it holds that 
	\begin{align*}
	E[\Gamma_{\tau}|\mathcal{F}_t](\omega)
	&=\sum_n 1_{A^n}(\omega) E_{P^t}[f^n_{\sigma^n}] 
	\geq \sup_{\sigma\in\mathcal{T}^t} E_{P^t}[\Gamma^{\omega,t}_{\sigma}]-\varepsilon
	\end{align*}
	for $P$-almost all $\omega\in\Omega$.
	Further, it can be checked that $\tau$ is a stopping time.
	Hence, under the assumption made on $\Gamma$, the second claim follows.
	
	We are left to argue why this assumption is not restrictive.
	First, by tightness of $P^t$, for every $\varepsilon>0$, there is some compact $K\subset\Omega^t$ for which 
	\[|E_{P^t}[\Gamma^{t,\omega}_{\sigma}] - E_{P^t}[\Gamma^{t,\omega}_{\sigma} 1_K]| \leq \varepsilon\]
	uniformly over all $\sigma\in\mathcal{T}^t$.
	Now note that $\omega\mapsto \Gamma^{t,\omega}1_K$ is a function with values in the separable space $C(K\times[t,T],\mathbb{R})$.
	Thus, it can be approximated uniformly by functions of the form $\sum_n f^n 1_{A^n}(\omega)$.
\end{proof}

\begin{proof}[Proof of Theorem \ref{thm:stopping}]
	By \cite[Proposition 2.3]{Karoui_reflected} one has that $S$ is the value process of the solution of a `reflected BSDE' 
	with barrier $\Gamma$. 
	Recall that a triple $(Y,Z,K)$ (where $Y$ and $K$ are adapted, $Z$ progressive and $K$ continuous, increasing with $K_0=0$) solves a reflected BSDE with barrier $\Gamma$ if
	\begin{equation*}
		\begin{cases}
			Y_t = F + \int_t^Tg_s(Y_s,Z_s)\,ds + K_T - K_t - \int_t^TZ_s\,dW_s\\
			Y_t \ge \Gamma_t \quad \text{for all } t \in [0,T]
		\end{cases}
	\end{equation*}
	and it holds $\int_0^T(Y_t - \Gamma_t)\,dK_t=0$.
	Therefore, $S$ has continuous paths.
	
	Let $t\in[0,T]$ be fixed and denote by $N$ the set of all $\omega\in\Omega$ such that 
	\[ S_t(\omega):=\mathop{\mathrm{ess.sup}}_{\tau\in\mathcal{T}_t} E[ \Gamma_\tau |\mathcal{F}_t](\omega)
	\neq \sup_{\sigma\in\mathcal{T}^t} E_{P^t}[\Gamma^{t,\omega}_{\sigma}]. \]
	By Lemma \ref{lem:ess.sup.sup} one has $P(N)=0$.
	Fix $\omega,\eta\in N^c$. 
	
	For every $s\in[t,T]$ and $\gamma\in\Omega^t$, unwrapping the definition of $X^{t,\omega}$ yields
	\begin{align*}
	|\Gamma^{t,\omega}_{s}(\gamma)- \Gamma^{t,\eta}_{s}(\gamma)|
	&= |\Gamma_{s}(\omega\otimes_t\gamma)- \Gamma_{s}(\eta\otimes_t\gamma)| \\
	&\leq L_\Gamma \| \omega\otimes_t\gamma - \eta\otimes_t\gamma\|_\infty \\
	&\leq L_\Gamma \|\omega-\eta\|_\infty.
	\end{align*}
	Thus, for every $\sigma\in\mathcal{T}^t$, it holds that
	\begin{align*}
	| E_{P^t}[\Gamma^{t,\omega}_{\sigma}] -E_{P^t}[\Gamma^{t,\eta}_{\sigma}] |
	&\leq  L_\Gamma \|\omega-\eta\|_\infty.
	\end{align*}
	In particular, as $\omega,\eta\in N^c$, this implies that
	\[|S_t(\omega)-S_t(\eta)|\leq L_\Gamma \|\omega-\eta\|_\infty.\]
	Hence, the function $\gamma\mapsto S_t(\gamma)$ is $P$-almost surely a $L_\Gamma$-Lipschitz function.
	The claimed functional inequality for (the law of) $S$ follows again from Lemma \ref{lem:stab push-forw} and \cite[Theorem 3.1]{Fey-Ust04}.
\end{proof}

\begin{corollary}
\label{cor:stopping bsde}
	Let $(Y,Z)$ solve \eqref{eq:bsde} and let $\Gamma\colon[0,T]\times C([0,T],\mathbb{R}^m)\to\mathbb{R}$ be an adapted process with continuous paths such that $\Gamma_t$ is $L_\Gamma$-Lipschitz for every $t\in[0,T]$.
	Under the conditions of Theorem \ref{thm:bsde.multi.dim} (in case $m\geq 1$) or Theorem \ref{thm:bsde.1.dim} (in case $m=1$), the process
	\[ S_t:=\mathop{\mathrm{ess.sup}}_{\tau  \text{ is stopping time, } t\leq \tau\leq T} E[\Gamma_{\tau}(Y)|\mathcal{F}_t] \]
	has continuous paths and its law satisfies $T_2(C)$ with $C= 2L_Y^2$ and $L_Y=L_F + TL_ge^{TL_g}$ (in case of Theorem \ref{thm:bsde.multi.dim}) or $L_Y= L_F$ (in case of Theorem \ref{thm:bsde.1.dim}).
\end{corollary}
\begin{proof}
	Set $\Gamma':=\Gamma\circ Y$.
	Then $\Gamma'$ is still adapted and it has continuous paths.
	Moreover, from (the proof of) Theorem \ref{thm:bsde.multi.dim} (resp.\ Theorem \ref{thm:bsde.1.dim}) it follows that $\Gamma'_t$ is Lipschitz continuous with a constant not depending on $t$.
	The claim now follows from Theorem \ref{thm:stopping}.
\end{proof}

\section{The proof of Theorem \ref{thm:bsde.1.dim}}
\label{sec:ineq bsde}

Recall the definition of the convex conjugate
\begin{align}
\label{eq:dual.rep.g}
g_t^\ast(q)=\sup_{z\in\mathbb{R}^d} (qz-g_t(z))
\quad\text{for every }  t\in[0,T] \text{ and } q\in\mathbb{R}^d.
\end{align}

\begin{remark}
	The proof of Theorem \ref{thm:bsde.1.dim} uses \cite[Lemma 5.1]{boue-dupuis-ramon}, which assumes
	\begin{equation}
	\label{eq:assum blt}
		\lim_{|q|\to +\infty}\inf_{t \in [0,T]}\frac{g^*_t(q)}{|q|} = +\infty \quad \text{and}\quad \int_0^T \Big|\sup_{|q|\le r}g^*_t(q) \Big|\,dt<\infty.
	\end{equation}
	When the equation \eqref{eq:bsde} admits a solutions, the conclusion of Theorem \ref{thm:bsde.1.dim} remains valid if the assumptions (B) and  (C) are replaced by the assumptions	\eqref{eq:assum blt}.
	These are in fact weaker assumptions, as the following lemma shows.
\end{remark}

\begin{lemma}
\label{lem:relation.g.g.ast}
	Assume that $\sup_{t\in[0,T]} g_t(z)<\infty$ for every $z\in\mathbb{R}^d$.
	Then it holds that $\inf_{t\in[0,T]} g_t^\ast(q)/|q|\to\infty$ as $|q|\to\infty$.
	
	On the other hand, assume that $\inf_{t\in[0,T]} g_t(z)/|z|\to\infty$ as $|z|\to\infty$.
	Then it holds that $\sup_{t\in[0,T]} g_t^\ast(q)<\infty$ for every $q\in\mathbb{R}^d$.
\end{lemma}
\begin{proof}
	To show the first claim, let $m\geq 0$ be arbitrary.
	As $z\mapsto \sup_{t\in[0,T]} g_t(z)$ is convex and real valued, it is continuous.
	Hence, there exists $c>0$ such that $g_t(z)\leq c$ for all $|z|\leq m$ and all $t\in[0,T]$.
	Plugging the choice $z:=mq/|q|$ in \eqref{eq:dual.rep.g}  implies
	\[ \inf_{t\in[0,T]} \frac{g_t^\ast(q)}{|q|}
	\geq \inf_{t\in[0,T]}  \Big( \frac{qz}{|q|}-\frac{g_t(z)}{|q|} \Big)
	\geq m -\frac{c}{|q|}
	\to m\]
	as $|q|\to\infty$.
	Since $m>0$ was arbitrary, this implies the first claim.
	
	To show the second claim, let $q\in\mathbb{R}^d$ be arbitrary. 
	We distinguish between small and large $z$ in the representation \eqref{eq:dual.rep.g}.
	By assumption there is $c>0$ such that $g_t(z)\geq  2|q| |z|$ for all $t\in[0,T]$ and $|z|\geq c$.
	For such $z$ one has $qz-g_t(z)\leq |z||q| - 2 |q||z|\leq 0$.
	On the other hand, $g\geq 0$ implies that $qz-g_t(z)\leq c|q|$ for all $|z|\leq c$.
	This show the second claim.
\end{proof}

\begin{proof}[Proof of Theorem \ref{thm:bsde.1.dim}]
	In a first step we focus on the case where $F$ is bounded.

	It follows from the condition (B) and \cite[Theorem 2.3]{kobylanski01} that equation \eqref{eq:bsde} admits a solution $(Y,Z)$.
	Let  $C_b(\Omega)$ denote the space of bounded continuous functions on $\Omega$.
	For any generator $h$, consider the functional $\rho^h\colon C_b(\Omega)\to\mathbb{R}$ which maps the terminal condition $F$ to $Y_0$, where $Y$ is the solution of \eqref{eq:bsde} with $g$ substituted by $h$.
	By \cite[Theorems 2.1 and 2.2]{BDH10} and convexity of $g$ one has 
	\begin{equation}
	\label{eq:dual rep Yn}
	 	Y_0 
	 	= \rho^{g}(F)
	 	=\sup_{q} E_Q\Big[ F - \int_0^T g^\ast_u(q_u)\,du \Big],
	\end{equation} 
	where the supremum is taken over all progressive and square integrable processes $q$ with values in $\mathbb{R}^d$, $Q$ a probability measure absolutely continuous with respect to $P$ and with density 
	\[ \frac{dQ}{dP}:=\exp\Big(\int_0^Tq_u\,dW_u - \frac12\int_0^T|q_u|^2\,du\Big).\]
	In particular it follows that 
	\begin{align}
	\label{eq:bsde.operator.lipschitz}
	|\rho^{g}(F)-\rho^{g}(G)|\leq \sup_{\omega\in\Omega} |F(\omega)-G(\omega)|
	\end{align}
	for all bounded functions $F,G\colon\Omega\to\mathbb{R}$.
	
	As in the proofs of Theorems \ref{thm:bsde.multi.dim} and Theorem \ref{thm:stopping} we use shifts of paths, however, this time defined with intrinsic scaling:
	For $t\in[0,T)$ and $\omega,\eta\in \Omega$ with $\eta(0)=0$, define $\omega\oplus_t \eta\in\Omega$ via
	\[(\omega\oplus_t \eta)_s := \omega({t\wedge s}) + \sqrt{T-t}\cdot \eta\Big(\frac{s - t}{T - t}\Big)1_{[t,T]}(s).\]
	Only Brownian motion will be plugged in as the second argument, hence  $\omega\oplus\eta$ only needs to be defined for paths $\eta$ which start in $0$. 
	However, to be formally correct, one can define $\omega\oplus_t \eta:=\omega\oplus_t (\eta-\eta(0))$ for all paths $\eta$ which do not start at 0.
	Moreover, define $g^{(t)}\colon [0,T]\times\mathbb{R}^d\to\mathbb{R}$ by
	\[g^{(t)}_s(z) := (T-t)g_{t+s(T-t)}\Big(\frac{z}{\sqrt{T -t}}\Big).\]
	
	Since $(Y,Z)$ is the unique solution of \eqref{eq:bsde} with generator $g$, it then follows from \cite[Theorem 2.2]{BDH10}, \cite[Theorem 4.5]{tarpodual} and \cite[Lemma 5.1]{boue-dupuis-ramon} that 
	\begin{align*}
		Y_t(\omega) = \rho^{g^{(t)}}(F(\omega \oplus_t \cdot))
	\end{align*}
	for $P$-almost all $\omega\in\Omega$ and every $t\in[0,T)$.
	Using the $1$-Lipschitz continuity of the operator $\rho^{g^{(t)}}(\cdot)$ shown in \eqref{eq:bsde.operator.lipschitz}
	it follows that
	\begin{align*}
	|Y_t(\omega)-Y_t(\eta)|
	&\leq \sup_{\gamma\in\Omega} | F(\omega \oplus_t \gamma) - F(\eta \oplus_t \gamma) | \\
	&\leq L_F  \sup_{\gamma\in\Omega} \| \omega \oplus_t \gamma - \eta \oplus_t \gamma \|_{\infty} \\
	&\leq L_F \|\omega-\eta\|_\infty
	\end{align*}
	for $P$-almost all $\omega,\eta\in \Omega$ and every $t\in[0,T)$.
	As $Y_T=F$ and $Y$ has $P$-almost surely continuous paths, we conclude that $Y\colon\Omega\to\Omega$ is $L_F$-Lipschitz $P$-almost surely.
	It thus follows by \cite[Theorem 3.1]{Fey-Ust04} and Lemma \ref{lem:stab push-forw} that the law $\mu^y$ of $Y$ satisfies $T_2(2L_F^2)$.

	In case that $F$ is not bounded it follows by Lipschitz continuity of $F$ that it has exponential moments. Denote by $Y^n$ the solution to $\eqref{eq:bsde}$ with $F$ replaced by $F\wedge n$ for each $n\in\mathbb{N}$.
	As $F\wedge n$ is $L_F$-Lipschitz, by the above, $Y^n\colon\Omega\to\Omega$ is $L_F$-Lipschitz, and it follows from stability of BSDE with terminal conditions having exponential moments (see \cite[Proposition 7]{Bri-Hu08}) that $Y^n\to Y$  $P\otimes dt$-almost surely (where $(Y,Z)$ is the solution of \eqref{eq:bsde}).
	As Lipschitz continuity is stable under pointwise convergence, $Y$ remains $L_F$-Lipschitz and the claim again follows from \cite[Theorem 3.1]{Fey-Ust04} and Lemma \ref{lem:stab push-forw}.
\end{proof}

\begin{remark}[Supersolutions]
	The condition (B) in Theorem \ref{thm:bsde.1.dim} serves as a guarantee that the (one-dimensional) BSDE with generator $g$ admits a solution $(Y,Z)$ such that $Y$ satisfies the representation \eqref{eq:dual rep Yn}.
	Without that condition, the BSDE still admits a unique \emph{minimal supersolution} $(\bar Y, \bar Z)$ in the sense of \cite{DHK1101}, and it follows from \cite{tarpodual} that $\bar Y$ satisfies the representation \eqref{eq:dual rep Yn}.
	Therefore, the proof of Theorem \ref{thm:bsde.1.dim} shows that the law of $\bar Y$ satisfies $T_2(C_y)$ with $C_y=2L_F^2$.
\end{remark}

\begin{corollary}
\label{cor: non positiv}
	Assume that
	\begin{enumerate}[(A)]
	\item
	 $g\colon[0,T]\times \mathbb{R}^d\to \mathbb{R}$ is Borel measurable, convex in the last variable, satisfies (B) in Theorem \ref{thm:bsde.1.dim} and there is $b\in\mathbb{R}$ and bounded Borel $a\colon [0,T]\to \mathbb{R}^d$ such that
	$g_t(z)\ge a_tz + b$. 
	\item 
	$F\colon\Omega\to \mathbb{R}$ is bounded from below and $L_F$-Lipschitz continuous.
	\end{enumerate}
	Then \eqref{eq:bsde} admits a unique solution $(Y,Z)$ and 
	\[ \text{the law } \mu^y \text{ of } Y \text{ satisfies } T_2(C_y)\]
	with $C_y := 2L_F^2$.
\end{corollary}
\begin{proof}
	Since $(Y,Z)$ satisfies \eqref{eq:bsde}, we have
	\begin{align*}
		Y_t+tb &= F+Tb + \int_t^Tg_u(Z)-(a_uZ_u + b)\,du - \int_0^TZ_u(\,dW_u - a_u\,du).
	\end{align*}	
	By Girsanov's theorem, the process $\tilde W:=W -\int_0^ta_u\,du$ is a Brownian motion under the probability measure $\tilde P$ with density 
	\[ \frac{d\tilde{P}}{dP}:=\exp\Big(-\int_0^Ta_u\,dW_u - \frac12\int_0^T|a_u|^2\,du\Big).\]
	Thus, putting 
	\begin{align*}
		\tilde Y_t := Y_t + tb,\quad \tilde F := F + Tb\quad \text{and}\quad \tilde g_t(z):= g_t(z) - (a_tz +b),
	\end{align*}
	it holds that $(\tilde Y, Z)$ solve equation \eqref{eq:bsde} driven by the Brownian motion $\tilde W$ with generator $\tilde g$ and terminal condition $\tilde F$.
	In oder words,
	\begin{equation*}
		\tilde Y_t = \tilde F + \int_t^T\tilde g_u(Z_u)\,du - \int_t^TZ_u\,d\tilde W_u\quad \tilde P\text{-a.s.}
	\end{equation*}
	Observe that in this case, the function $\tilde g$ is convex, positive and satisfies the growth conditions $\tilde g_t(z)\le C(1 + |z|^2)$ for some $C$ and $\lim_{|q|\to \infty}\inf_{t\in[0,T]}\tilde{g}(q)/|q|=+\infty$.
	As argued in the proof of Theorem \ref{thm:bsde.1.dim}, the process $\tilde Y$ satisfies $\tilde Y_t(\omega) = \rho^{\tilde g^{(t)}}(\tilde F(\omega\oplus_t\cdot)$ for a $1$-Lipschitz continuous operator (depending on $\tilde P$).
	In particular, $\omega\mapsto \tilde Y(\omega)$ (and therefore $\omega\mapsto Y(\omega)$) is $L_F$-Lipschitz continuous and thus, the result follows from \cite[Theorem 3.1]{Fey-Ust04} and Lemma \ref{lem:stab push-forw}.
\end{proof}

\subsection{Portfolio optimization}
\label{sec:portfolio}
Let us come back to the quantitative finance application alluded to in Examples \ref{exa:finance} in the introduction.
Consider a market with $m$ stocks whose prices are given by the $m$-dimensional process  $S$ following the Black-Scholes model
\[ dS_t = S_t(b_t\,dt + \sigma_t\,dW_t)\]
with $b$ and $\sigma$ two bounded functions of $t$ with appropriate dimensions (recall that $W$ is a $d$-dimensional Brownian motion).
A basic task in quantitative finance consists of optimizing the expected utility of a given claim by dynamic trading.
Concretely, let us fix the exponential utility $U(x):=-\exp(-\alpha x)$ for some $\alpha >0$ modeling the investor's preferences and a claim $F\colon\Omega\to\mathbb{R}$.
Then the problem in question reads as
\begin{align}
\label{eq:util max}
	u_t:= \frac{1}{\alpha} \log\Big( - \mathop{\mathrm{ess\,sup}}_{p} E\Big[-\exp\Big( - \alpha \Big( \int_t^T p_s\,(dW_t+\theta_tdt) - F\Big)\Big) \Big| \mathcal{F}_t\Big]\Big).
\end{align}
	Here $\mathcal{A}$ is the set of admissible strategies, i.e.\ the set of all predictable processes $p$ with values in a convex and closed set $\mathbb{A}\subseteq \mathbb{R}^d$ for which $E[\int_0^T|p_t|^2\,dt]<\infty$ and $\{ \exp( -\alpha \int_0^\tau p_t\,(dW_t+\theta_tdt)) : \tau\text{ is stopping time}\}$ is uniformly integrable, and $\theta_t:= \sigma^{tr} (\sigma_t\sigma^{tr}_t)^{-1}b_t$.
In particular, $\sigma_t\sigma^{tr}_t$ is invertible and we assume moreover that $\theta$ is a bounded function of time only.
	 \begin{theorem}[\text{\cite[Theorem 7 and Proposition 9]{Hu-Imk-Mul}}]
	Assume that $F$ is bounded.
	There is an admissible portfolio $p^\ast$ which is optimal for all $t$ simultaneously (i.e.\  $p^\ast$ achieves \eqref{eq:util max} for every $t\in[0,T]$).
	
	Moreover, defining $g$ by $g_t(\omega,y,z)= (\alpha/2)\mathrm{dist}^2(z+\theta_t/\alpha, \mathbb{A})-z\theta_t - |\theta_t|^2/(2\alpha)$, the pair $(u, p^\ast-\theta/\alpha)$ solves \eqref{eq:bsde} with generator $g$ and terminal condition $F$.
\end{theorem}

Since the constraint set $\mathbb{A}$ is convex, the generator $g$ satisfies the conditions of Corollary \ref{cor: non positiv}.
Thus, if $F$ is Lipschitz continuous, then the law of $u$ satisfies $T_2(C)$ for some constant $C>0$.
If $\mathbb{A}=\mathbb{R}^d$ and we additionally assume that $F$ is  Malliavin differentiable with Lipschitz continuous Malliavin differentials, then by Corollary \ref{cor:actual T2 for z} below the law of $p^\ast$ satisfies $T_2$ as well.

These imply that (under the above assumptions made on $F$ and $\mathbb{A}$), convergence of empirical measure as in Corollary \ref{cor:wasserstein.conc} can be deduced, or the Gaussian concentration
	\begin{align*}
	 	P(|f(u_t)- E[f(u_t)]| \ge  x) &\le 2\exp(-cx^2/L_f^2), \\ 
	P(|g(p^*_t)- E[g(p^*_t)]| \ge  x) &\le 2\exp(-cx^2/L_g^2)	
	\end{align*}
	for every $x>0$ and $L_f$-Lipschitz continuous functions $f$ and $g$ of appropriate dimensions.

\section{Transportation inequalities for the control process}
\label{sec:ineq z}

This section presents (modest) results on transportation inequalities for the control process $Z$.
Its main finding is Corollary \ref{cor:actual T2 for z} which shows that for a linear equation the law of $Z$ satisfies the $T_2$ inequality.

We use the notation $D^i_s\xi$ for the Malliavin derivative of the random variable $\xi$ in the direction of the $i^{\mathrm{th}}$ Brownian motion.

\begin{lemma}
\label{lem:z bounded}
	In addition to the assumptions of Theorem \ref{thm:bsde.multi.dim}, assume that
	\begin{enumerate}[(A)]
	\item[(C)]
		We have $g_\cdot(0,0,0)\in L^2([0,T])$,	for every $(t,y,z)$, the function $\omega\mapsto g_t(\omega, y,z)$ is Malliavin differentiable and it holds
	\begin{equation*}
		|D_u^i g_t(\omega,y^1,z^1) - D_u^ig_t(\omega,y^2, z^2)| \le K_u(t)\left(|y^1 - y^2| +  |z^1- z^2| \right)
	\end{equation*}
	for every $ \omega \in \Omega$, $y^1, y^2\in \mathbb{R}^m$, $z^1, z^2 \in \mathbb{R}^{m\times d}$, $t,u \in [0,T]$ and $i= 1,\dots ,d$ and
	for some $\mathbb{R}_+$-valued adapted process $(K_u(t))_{t,u \in [0,T]}$ such that we have $\int_0^TE[(\int_0^T|K_u(t)|^2\,dt)^2]\,du<\infty$.
	\item[(D)] The function $F$ is Malliavin differentiable and $t\mapsto D_t^iF$ is continuous.
	\end{enumerate}
	Then $Z$ has continuous paths and is bounded.
	In particular, the law $\mu^z$ of $Z$ satisfies 
	\begin{equation}
	\label{eq:ineq z}
		\mathcal{W}_2(\mu^z, \nu) \le C_zH(\nu|\mu^z)^{1/4}\quad \text{for all } \nu \in \mathcal{P}(\Omega),
	\end{equation}
	with $C_z:= 2\big(1 + \big(mL_F^2e^{(L_g + 1)^2T} + mL^2_gT  \big)^4\big)^{1/4}$.
\end{lemma}

\begin{remark}
	Notice that the Malliavin differentiability of $g$ and $F$ are consequences of the Lipschitz continuity assumptions made in Theorem \ref{thm:bsde.multi.dim}.
	The additional property needed in Lemma \ref{lem:z bounded} is the regularity of $D^i_ug$.
\end{remark}

\begin{proof}
	First assume that the function $g$ is continuously differentiable in $(y,z)$.
	Since $g$ is $L_g$-Lipschitz continuous, it follows by \cite[Proposition 3.2]{Che-Nam} that $g_t(\cdot,y,z)$ is Malliavin differentiable for every $t, y,z$ and $|D^i_sg_t(y,z)|\le L_g$ for all $i = 1, \dots, d$.
	Similarly, $F$ is Malliavin differentiable and for every $t\in [0,T]$, it holds that $|D^i_tF|\le L_F$ for all $i=1,\dots, d$.
	Thus, since $t\mapsto D_t^iF$ is continuous, it follows by \cite[Proposition 5.3]{karoui01}, that the process $Z$ has a version with continuous paths.
	Moreover, $Z$ is bounded, see \cite[Lemma 3.2]{Kup-Luo-Tang18}.
	In fact, it is shown therein that $Z$ satisfies
	\begin{equation}
	\label{eq:bound z}
		|Z_t|^2 \le mL_F^2e^{(2L_g + L_g^2 + 1)T} + mL^2_gT =: C.
	\end{equation}

	If $g$ is not continuously differentiable in $(y,z)$, let $g^n$ be a sequence of smooth functions converging to $g$ and denote by $(Y^n,Z^n)$ the solution of equation \eqref{eq:bsde} with generator $g^n$.
	Then, it follows for instance by \cite[Proposition 2.1]{karoui01} that $Z^n$ converges to $Z$ in $L^2(P\otimes dt)$.
	Therefore, $Z$ also satisfies \eqref{eq:bound z}.
	
	In particular, it has exponential moments of all orders.
	Thus, it follows by \cite[Corollary 2.4]{Bol-Vil05} that the law $\mu^z$ of $Z$ satisfies \eqref{eq:ineq z}.
\end{proof}

In the next corollary, we show that when the Malliavin derivative of the function $g$ is bounded, it does not need to be Lipschitz continuous in $z$ in order to have a transportation inequality for the law of $Y$.
In fact, the function $g$ can grow arbitrarily fast in its last variable.
\begin{corollary}
\label{cor:arb fast}
	Assume that $g:[0, T]\times \Omega \times \mathbb{R}^m \times \mathbb{R}^{m\times d} \to \mathbb{R}^m$ satisfies 
	\begin{itemize}
		\item[(A)] There is an increasing function $\varphi:\mathbb{R}_+\to \mathbb{R}_+$ such that for every $\omega^1, \omega^2 \in \mathcal{C}$, $y^1, y^2\in \mathbb{R}^m$, $z^1, z^2 \in \mathbb{R}^{m\times d}$ it holds
	\begin{equation*}
			|g_t(\omega^1, y^1, z^1) - g_t(\omega^2, y^2,z^2) |\le L_g\left( ||\omega^1 - \omega^2||_\infty+ |y^1 - y^2|  \right) + \varphi(|z^1|\vee |z^2|)|z^1 - z^2|.
		\end{equation*}
	\item[(B)]
	$F\colon\Omega\to \mathbb{R}^m$ is $L_F$-Lipschitz continuous.
	\item[(C)] Condition (C) in Lemma \ref{lem:z bounded} is satisfied.
	\end{itemize}
	Then, if $T$ is small enough the equation \eqref{eq:bsde} admits a unique solution $(Y,Z)$ and 
	$$
		\text{the law } \mu^y \text{ of } Y \text{ satisfies } T_2(C_y)
	$$ 
	with $C_y:= 2(L_F + T\max(L_g, \rho(Q)))^2 e^{2T\max(L_g, \varphi(\Lambda))}$, where $\Lambda$ is given by $\Lambda :=\big(2dm ( L_F^2 + TL_g^2 )\big)^{1/2}$.

	Furthermore, if the map $t\mapsto D_t^iF$ is continuous, then the law $\mu^z$ of $Z$ satisfies 
	\begin{equation*}
		\mathcal{W}_2(\mu^z, \nu) \le C_zH(\nu|\mu^z)^{1/4}\quad \text{for all } \nu \in \mathcal{P}(\Omega),
	\end{equation*}
	with $C_z:= 2\big(1 + \big(mL_F^2e^{(\max(L_g, \varphi(\Lambda)) + 1)^2T} + mT\max(L_g, \varphi(\Lambda))^2  \big)^4\big)^{1/4}$.
\end{corollary}
\begin{proof}
	As argued in the proof of Lemma \ref{lem:z bounded}, $F$ is Malliavin differentiable and has a derivative bounded by $L_F$.
	Therefore, the existence of a unique solution $(Y,Z)$ 
	follows from \cite[Theorem 3.1]{Kup-Luo-Tang18}, where it is further proved that the process $Z$ satisfies $|Z| \le \Lambda$,
	provided that $T\le \frac{\log(2)}{2L_g + \varphi^2(\Lambda)+1}$.
	The truncated function 
	\begin{equation*}
		\hat g_t(y,z) := \begin{cases}
			g_t(y,z) \quad \text{if } |z| \le \Lambda\\
			g_t(y,\Lambda z/|z|) \quad \text{if } |z| > \Lambda
		\end{cases}
	\end{equation*}
	is Lipschitz continuous with Lipschitz constant smaller than $\max(L_g, \varphi(\Lambda))$, and since $g_t(Y_t, Z_t) = \hat g_t(Y_t, Z_t)$ $P$-almost surely for every $t\in [0,T]$ we conclude by uniqueness that $(Y,Z)$ solves the equation \eqref{eq:bsde} with $g$ replaced by $\hat g$.
	Thus, the results follow by Theorem \ref{thm:bsde.multi.dim} and Lemma \ref{lem:z bounded}.
\end{proof}
	
\begin{remark}
	When the function $g$ is deterministic, it is automatically Malliavin differentiable and its Malliavin derivative is zero.
	In this case, all the conditions pertaining to the Malliavin derivative of $g$ in Lemma \ref{lem:z bounded} and Corollary \ref{cor:arb fast} are trivially satisfied.
	
	On the other hand, the smallness condition on the time horizon $T$ is necessary because $Y$ is a multidimensional process.
	It is well known that in the multidimensional case, and when the function $g$ is allowed to grow as fast as the quadratic function, backward SDE are typically ill-posed for arbitrary time horizons, see for instance \cite{Frei-dosReis11} for a discussion of this issue.
	The smallness condition is not necessary in one dimension.
\end{remark}

In Lemma \ref{lem:z bounded} and Corollary \ref{cor:arb fast} we derived a transportation inequality of the form $W_2(\mu^z,\nu)\le \phi(H(\mu^z|\nu))$ for the law of $Z$, with $\phi(x)= x^{1/4}$.
While this type of inequalities (extensively studies e.g.\ in \cite{Goz-Leo07}) allow to derive deviation inequalities, they do not tensorize or allow to derive other important inequalities as Poincar\'e inequality.

The next corollary provides a simple example under which Talagrand inequality holds for the law of $Z$.
It is the case of a linear equation.
\begin{corollary}
\label{cor:actual T2 for z}
	Assume that
	\begin{enumerate}[(A)]
	 	\item 
	 	\label{d1}
	 	$F\colon\Omega\to \mathbb{R}^m$ is Malliavin differentiable and its Malliavin derivatives $D^iF_t\colon\Omega\to \mathbb{R}^m$ are $L_F$-Lipschitz continuous.
	 	\item
	 	\label{d2} $g_t(\omega,y,z) = \alpha_t(\omega)+ \beta y + \gamma z$
for some constants $\beta,\gamma$ and a progressive, square integrable process $\alpha$ such that for each $t$, $\alpha_t$ is Malliavin differentiable and its derivative are $L_\alpha$-Lipschitz continuous.
	 \end{enumerate} 
	 	Let $(Y,Z)$ be the unique solution of equation \eqref{eq:bsde}.
	 Then, the 
	\begin{equation*}
		\text{the law } \mu^z  \text{ of } Z  \text{ satisfies } T_2(C_z) 
	\end{equation*}
	with $C_z: = 2(L_F+TL_G)^2e^{2TL_G}$ and $L_G:= \max(L_\alpha, |\beta|,|\gamma|)$.

	If in addition $F$ and $\alpha_t$ are respectively $L_F$- and $L_\alpha$-Lipschitz continuous, then 
	\begin{equation*}
		\text{the law } \mu^{y,z}  \text{ of } (Y,Z)  \text{ satisfies } T_2(C_{y,z}) 
	\end{equation*}
	with $C_{y,z}:= \max(C_y, C_z)$.
\end{corollary}
\begin{proof}
 	It follows from \cite{karoui01} that for every $t \in [0,T]$, the pair $(Y_t, Z_t)$ is Malliavin differentiable and a version of the derivatives $(D^i_sY_t, D^i_sZ_t)$, $i=1, \dots, d$  satisfies the linear equations
 	\begin{equation}
 	\label{eq:Malliavin bsde}
 		D^i_sY_t = D^i_sF + \int_t^T D_s^i\alpha_r + \beta D^i_sY_r + \gamma D^i_sZ_r\,dr- \int_t^TD^i_sZ_r\,dW_r, 
 	\end{equation}
 	for $i = 1,\dots, d$ and $D_t^iY_t = Z_t^i$.
 	Moreover, by (B), the function $G(\omega, y,z)$ given by
 	\begin{equation*}
 		G(\omega, y, z) := D_s^i\alpha_r(\omega) + \beta y + \gamma z
 	\end{equation*}
 	is Lipschitz continuous, with Lipschitz constant $L_G:= \max(L_\alpha, |\beta|,|\gamma|)$.
 	Thus, it follows by Theorem \ref{thm:bsde.multi.dim} (and its proof) that the process $(D_tY_t)_t$ is a Lipschitz continuous function of $W$, and its law satisfies $T_2(C_z)$.
 	Since $Z_t = D_tY_t$ $P\otimes dt$-almost surely, the first claim follows. 

 	If $\alpha$ and $F$ are Lipschitz continuous, then it follows by Theorem \ref{thm:bsde.multi.dim} that $Y = \phi(\omega)$ for some Lipschitz continuous function $\phi$.
 	Thus, the second claim follows by Lemma \ref{lem:stab push-forw} since $(Y,Z)$ is a Lipschitz continuous function of $W$.
\end{proof}

\section{Proof for SDEs}
\label{sec:sde}

\begin{proof}[Proof of Theorem \ref{thm:SDE}]
	It follows from the conditions (B) and (C) that the function $b$ is bounded.
	Thus, the existence of a unique strong solution follows from \cite[Theorem 4]{Zvonkin}.
	
	Define the following three functions
	\begin{align*}
	f_t(x)&:=\exp\Big(\int_{-\infty}^x-2\frac{b_t(a)}{\sigma_t\sigma'_t(a)}\,da \Big),\\
	F_t( x)&:= \int_0^xf_t(a)\,da, \quad\text{and}\\
	G_t(x)&:= [F_t(\cdot)]^{-1}(x)
	\end{align*}
	for all $(t,x) \in [0,T]\times \mathbb{R}$.
	Note that it follows from the integrability assumption \eqref{ass:sde.b.divided.sigma} that $F_t(\cdot)$ is bijective and therefore $G$ is well-defined.
	Moreover, as $b_\cdot(x)$ and $\sigma_\cdot(x)$ are differentiable by assumption \eqref{ass:sde.b} and \eqref{ass:sde.sigma}, it follows from \eqref{ass:sde.b.divided.sigma} that $f_\cdot(x)$ (and therefore $F_\cdot(x)$) is differentiable for every $x\in\mathbb{R}$.
	Further $F_t(\cdot)$ is differentiable with derivative $\partial_x F_t(x)=f_t(x)$ for every $(t,x)\in[0,T]\times\mathbb{R}$ and $\partial_x F_t(\cdot)$ is absolutely continuous w.r.t.~Lebesgue measure for every $t\in[0,T]$.
	Thus, it admits a weak derivative which we denote $\partial_{xx} F_t(\cdot)$.
	That is, $F$ belongs to the Sobolev space  $W^{1,2}_2([0,T]\times \mathbb{R})$.
	
	Putting $Y_t: = F_t( X_t)$, it follows from It\^o-Krylov's formula  \cite[Theorem 2.10.1]{Krylov} 	that
	\begin{align*}
	Y_t&=Y_0 + \int_0^t \partial_tF_s(X_s) +\frac{1}{2}\partial_{xx} F_s(X_s)\sigma_s\sigma_s'(X_s)\,ds +\int_0^t \partial_x F_s(X_s) b_s(X_s)\,dW_s \\
	&=Y_0 + \int_0^t\partial_t F_s( G_s( Y_s))\,ds + \int_0^t f_s( G_s( Y_s))\sigma_s( G_s( Y_s))\,dW_s.
	\end{align*}
		 
	By assumption \eqref{ass:sde.b.divided.sigma} one has $\exp(-c_1)\leq f\leq \exp(c_1)$.
	Hence, for every $t\in[0,T]$, both mappings 
	\[F_t(\cdot)
	\quad\text{ and }\quad
	G_t(\cdot) \text{ are } \exp(c_1)\text{-Lipschitz}.\]
	Further, as $|\partial_t f|$ is bounded by $\exp(c_1)c_3$, it follows that
	\[\partial_t F_t(\cdot)=\int_0^\cdot \partial_t f_t(a)\,da \quad\text{is } c_3\exp(c_1)\text{-Lipschitz}\]
	for every $t\in[0,T]$.
	Thus $\partial_tF_t( G(t, \cdot))$ is $c_3\exp(2c_1)$-Lipschitz for every $t\in[0,T]$.
	Moreover 
	\begin{align*}
	f_t(\cdot) &\quad\text{is }c_2\exp(c_1)\text{-Lipschitz},\\
	f_t(G_t(\cdot)) &\quad\text{is } c_2\exp(2c_1)\text{-Lipschitz, and}\\
	\sigma_t(G_t(\cdot)) &\quad\text{is } L_\sigma\exp(c_1)\text{-Lipschitz}
	\end{align*} 
	for every $t\in[0,T]$.
	As $f$ is bounded by $\exp(c_1)$ and $\sigma$ by $\|\sigma\|_\infty$, one obtains that the product 
	\[f_t( G_t( \cdot))\sigma_t( G_t(\cdot)) \quad\text{is } \|\sigma\|_\infty c_2\exp(2c_1) + \exp(2c_1)L_\sigma\text{-Lipschitz}\]
	for every $t\in[0,T]$.

	Therefore, \cite[Theorem 1]{Uestuenel12} shows that the law $\mu^y$ of $Y$ satisfies $T_2(C_y)$ with 
	\[C_y=6\exp(15 \max( c_3\exp(2c_1, \|\sigma\|_\infty c_2\exp(2c_1) + \exp(2c_1)L_\sigma^2)).\]
	Moroever, by the above, $X=G(Y)$ and the mapping $G$ is $\exp(c_1)$-Lipschitz.
	An application of Lemma \ref{lem:stab push-forw} shows that $\mu^x$ satisfies $T_2(\exp(c_1) C_y)$.
\end{proof}

\begin{example}
	Let us come back to the Langenvin dynamic presented in the introduction, see \eqref{eq:langenvin}. 
	These equations play a fundamental role notably  in the modeling of physical phenomena, see for instance the survey of \cite{Schuss80} for applications in Physics, and the more recent mathematical treatments \cite{JKO98,Conforti2018}.
It follows by Theorem \ref{thm:SDE} that if $U'$ is integrable and bounded, then the law of $X$ satisfies $T_2(C)$ for some $C>0$.

It is well-known that if $U$ grows fast enough for $e^{-\lambda U}$ to be integrable, (and it holds $xU'(x)\ge c_1x^2 - c_2$ for some $c_1,c_2\ge0$) then the measure $\mu$ on $\mathbb{R}$ with density (with respect to Lebesgue measure) $e^{-\lambda U(x)}/\int_{\mathbb{R}}e^{-\lambda U(x)}\,dx$ is the unique stationary measure of the solution of the Langenvin equation, see e.g.\ \cite[Lemma 1.2]{Lac-Shk-Zha19}.
Since $U'$ does not depend on time, the constant $C_x$ in Theorem \ref{thm:SDE} is time independent.
Therefore, it follows by Theorem \ref{thm:SDE} that the measure $\mu$ satisfies $T_2(C)$.

It is well-known, see e.g.\ \cite[Theorem 5.2]{Ledoux01}, that log-concave measures\footnote{A probability measure $\mu$ is said to be log-concave if its density with respect to Lebesgue measure is of the form $e^{-U}$ with $U$ convex.} satisfy the quadratic transportation inequality.
Notice that in the above arguments we do not require $U$ to be convex, but simply a differentiable function on the real line.
Nevertheless, a standard tensorization argument for $T_2$ inequalities allows to extend the argument to measures on $\mathbb{R}^d$ with density of the form $e^{-\sum_{i=1}^dU(x_i)}$.
\end{example}

\section{Logarithmic-Sobolev inequality}
\label{sec:lsi}

In this final section we prove the logarithmic-Sobolev inequality for the law of $Y_t$.
This requires some notational preparations.

Let $H$ denote the Cameron-Martin space
\[H:= \Big\{h \in \Omega: h \text{ is absolutely continuous, $h_0=0$ and } \int_0^T|\dot h_s|^2\,ds<\infty\Big\}.\]
Then $H$ becomes a Hilbert space when equipped with the inner product $\langle h,g\rangle_{H}:= \int_0^T\dot h_s\dot g_s\,ds$ for $h,g \in H$ and associated norm $\|h\|_{H}:= \langle h,h\rangle_{H}^{1/2}$.
Let $F\colon\Omega\to \mathbb{R}$ be Malliavin differentiable, with $D_tF \in L^2$ for all $t\in[0,T]$.
Given $h \in H$,
consider the directional derivative
\begin{equation*}
	D^hF(\omega) := \lim_{\varepsilon \downarrow 0} \frac{F(\omega + \varepsilon h) - F(\omega)}{\varepsilon}.
\end{equation*}
This defines a continuous linear operator on $H$, so that by Riesz representation theorem, there is a map $\omega\mapsto \bar D_sF(\omega)$ with values in $H$ such that
\begin{equation*}
	D^hF(\omega) = \langle \bar D F(\omega), h\rangle_{H}.
\end{equation*}
It is well-known that $\bar D F = DF$ $P$-almost surely, see e.g.\ \cite[Remark B.6.2]{UsZa1}.
In particular, as $H$ is separable, this implies that $|DF|_H=\sup_{h\in H \text{ s.th.\ } |h|_H\leq 1} |D^hF|$ $P$-almost surely.
The following Lemma is in spirit the same as Lemma \ref{lem:stab push-forw}, stating that Lipschitz transformations preserve the the log-Sobolev inequality.
It is likely to be known, but we could not find a reference and therefore provide its proof.
Note however that the finite dimensional case is given in \cite[Section 1]{Col-Fig-Jha17}.

\begin{lemma}
\label{lem:stabil LSI}
	Let $\psi\colon (\Omega,\|\cdot\|_{\infty}) \to \mathbb{R}^m$ be $L_\psi$-Lipschitz continuous.
	Then $\psi_*P$ satisfies $LSI(2 T L^2_\psi)$.
\end{lemma}
\begin{proof}
	In a first step, note that by Lipschitz continuity of $\psi$, it is Malliavin differentiable (with derivative bounded by $L_\psi$), see e.g.\ \cite[Proposition 3.2]{Che-Nam}.
	Let $f\colon\mathbb{R}^m\to \mathbb{R}$ be differentiable. 
	We need to show that $\mathrm{Ent}_{\psi_\ast P } (f)\leq 2 T L_\psi^2 \int  |\nabla f| \,d(\psi_\ast P)$.
	
	To that end, let $\omega\in\Omega$ and $h\in H$ be arbitrary and note that
	\begin{align*}
	&\limsup_{\varepsilon \downarrow 0} \Big| \frac{f(\psi(\omega + \varepsilon h)) - f(\psi(\omega))}{\varepsilon} \Big| \\
	&\leq \limsup_{\varepsilon\downarrow 0} \Big| \nabla f(\psi(\omega)) \cdot\frac{\psi(\omega + \varepsilon h) - \psi(\omega)}{\varepsilon}  \Big|
	\leq |\nabla f(\psi(\omega))| \cdot L_\psi \|h\|_\infty
	\end{align*}
	by $L_\psi$-Lipschitz continuity of $\psi$ and the Cauchy-Schwarz inequality.
	Further, H\"older's inequality implies that $\|h\|_\infty\leq\sqrt{T} |h|_H$, hence 
	\[D^h (f\circ\psi)(\omega)\leq L_\psi \sqrt{T} \cdot |\nabla f(\psi(\omega))| \cdot |h|_H\]
	for every $h\in H$.
	By the discussion preceding the lemma, this therefore implies that
	\[ |D(f\circ\psi)|_H
	\leq  L_\psi \sqrt{T} \cdot |\nabla f \circ \psi| \]
	$P$-almost all surely.
		
	Now notice that $\mathrm{Ent}_P(f\circ\psi)=\mathrm{Ent}_{\psi_\ast P } (f)$ by the transformation lemma.
	Hence, as the Wiener meausre $P$  satisfies $LSI(2)$ by \cite{Gross75} (see also \cite{Cap-Hsu-Led97} and \cite[Theorem 1.1]{Gou-Wu06} for a formulation using the Malliavin gradient as ours), it follows  that
	\begin{align*}
		\mathrm{Ent}_{\psi_\ast P } (f)
		&=\mathrm{Ent}_P(f\circ\psi) \\
		&\le 2 \int_\Omega |D (f\circ \psi) |_{H}^2\,d P	
		\leq 2 T L_\psi^2 \int_{\mathbb{R}^m}  |\nabla f| \,d(\psi_\ast P).
	\end{align*}
	This proves the claim.
\end{proof}

\begin{proof}[Proof of Theorem \ref{thm:lsi y}]
	Recall from the proof of Theorem \ref{thm:bsde.multi.dim} that $Y\colon\Omega\to\Omega$ is $L_Y$-Lipschitz with $L_Y=\sqrt{C_y / 2}$ (where $C_y$ is the constant given in that theorem).
	In particular $Y_t\colon\Omega\to\mathbb{R}^m$ remains $L_Y$-Lipschitz.
	The proof is completed by an application of Lemma \ref{lem:stabil LSI}.
\end{proof}

\section*{Acknowledgements}
D.\ Bartl greatfully acknowledges support by the Austrian Science Fund under FWF project P28661.

\end{document}